\documentclass[12pt,twoside]{amsart}
\usepackage[english]{babel} 
\usepackage[T1]{fontenc}
\usepackage[utf8]{inputenc}
\usepackage{amsmath, amssymb, amsthm}            
\usepackage{amstext, amsfonts, a4}

\usepackage{graphics,graphicx,subfigure}
\usepackage{caption} 
\usepackage{xcolor}

\usepackage{lmodern}
\usepackage{xfrac}
\usepackage{faktor}
\theoremstyle{plain}
	\newtheorem{Thm}{Theorem}[section] 
	\newtheorem{Prop}[Thm]{Proposition}        
	\newtheorem{Lem}[Thm]{Lemma}           
	\newtheorem{Coro}[Thm]{Corollary}

\theoremstyle{definition}
	\newtheorem{Def}[Thm]{Definition}

\theoremstyle{remark}
	\newtheorem{Rem}[Thm]{Remark}

\def\emptyset{\varnothing}

\def\NN{{\mathbb N}}
\def\ZZ{{\mathbb Z}}
\def\cP{{\mathcal P}}
\def\cQ{{\mathcal Q}}
\def\cR{{\mathcal R}}

\def\mfA{{\mathfrak A}}
\def\mfS{{\mathfrak S}}

\newcommand{\Ker}{\operatorname{Ker}}
\newcommand{\Id}{\operatorname{Id}}

\newcommand{\eps}{\varepsilon}
\newcommand{\IET}{\operatorname{IET}}

\newcommand{\Homeo}{\operatorname{Homeo}}
\newcommand{\PC}{\operatorname{PC}}

\newcommand{\PAff}{\operatorname{PAff}}
\newcommand{\setm}{\smallsetminus}

\begin{document}

\title{Signature for piecewise continuous groups}
\author{Octave Lacourte}
\date{February 25, 2020}
\maketitle

\begin{abstract}
Let $\widehat{\PC^{\bowtie}}$ be the group of bijections from $\mathopen{[} 0,1 \mathclose{[}$ to itself which are continuous outside a finite set. Let $\PC^{\bowtie}$ be its quotient by the subgroup of finitely supported permutations. \\
\indent We show that the Kapoudjian class of $\PC^{\bowtie}$ vanishes. That is, the quotient map $\widehat{\PC^{\bowtie}} \rightarrow \PC^{\bowtie}$ splits modulo the alternating subgroup of even permutations. This is shown by constructing a nonzero group homomorphism, called signature, from $\widehat{\PC^{\bowtie}}$ to $\faktor{\ZZ}{2\ZZ}$. Then we use this signature to list normal subgroups of every subgroup $\widehat{G}$ of $\widehat{\PC^{\bowtie}}$ which contains $\mfS_{\mathrm{fin}}$ and such that $G$, the projection of $\widehat{G}$ in $\PC^{\bowtie}$, is simple.
\end{abstract}

\section{Introduction}

Let $X$ be the right-open and left-closed interval $\mathopen{[}0,1 \mathclose{[}$. We denote by $\mfS(X)$ the group of bijections of $X$ to $X$. This group contains the subgroup composed of all finitely supported permutations is denoted by $\mfS_{\mathrm{fin}}$. The classical signature is well-defined on $\mfS_{\mathrm{fin}}$ and its kernel, denoted by $\mfA_{\mathrm{fin}}$, is the only subgroup of index $2$ in $\mfS_{\mathrm{fin}}$. An observation, originally due to Vitali \cite{Vitali}, is that the signature does not extend to $\mfS(X)$.~\\

For every subgroup $G$ of $\faktor{\mfS(X)}{\mfS_{\mathrm{fin}}}$, we denote by $\widehat{G}$ its inverse image in $\mfS(X)$. The cohomology class of the central extension $$0\to \faktor{\ZZ}{2\ZZ}= \faktor{\mfS_{\mathrm{fin}}}{\mfA_{\mathrm{fin}}} \to \faktor{\widehat{G}}{\mfA_{\mathrm{fin}}} \to G \to 1$$
is called the Kapoudjian class of $G$; it belongs to $H^2(G,\faktor{\ZZ}{2\ZZ})$. It appears in the work of Kapoudjian and Kapoudjian-Sergiescu \cite{Kapoudjian2002,KapoudjianSergiescu2005}. The vanishing of this class means that the above exact sequence splits; this means that there exists a group homomorphism from the preimage of $G$ in $\mfS(X)$ onto $\faktor{\ZZ}{2\ZZ}$ which extends the signature on $\mfS_{\mathrm{fin}}$ (for more on the Kapoudjian class, see \cite[\S 8.C]{2019arXiv190105065C}). This implies in particular that $\faktor{\widehat{G}}{\mfA_{\mathrm{fin}}}$ is isomorphic to the direct product $G \times \faktor{\ZZ}{2\ZZ}$. One can notice that for $G=\faktor{\mfS(X)}{\mfS_{\mathrm{fin}}}$ we have $\widehat{G}=\mfS(X)$; in this case the Vitali's observation implies that the Kapoudjian class does not vanish.

 The set of all permutations of $X$ continuous outside a finite set is a subgroup denoted by $\widehat{\PC^{\bowtie}}$. The aim here is to show the following theorem:

\begin{Thm}\label{Theorem Existence of a nontrivial group homomorphism onto Z/2Z}
There exists a group homomorphism $\eps: \widehat{\PC^{\bowtie}} \rightarrow \faktor{\ZZ}{2 \ZZ}$ that extends the classical signature on $\mfS_{\mathrm{fin}}$.
\end{Thm}

\begin{Coro}\label{Corollary vanishing of the Kapoudjian class}
Let $G$ be a subgroup of $\PC^{\bowtie}$. Then the Kapoudjian class of $G$ is zero.
\end{Coro}

This solves a question asked by Y. Cornulier \cite[Question 1.15]{Cornulier_rea}.\\

The subgroup of $\widehat{\PC^{\bowtie}}$ consisting of all permutations of $X$ that are piecewise isometric elements is denoted by $\widehat{\IET^{\bowtie}}$ and the one consisting of all piecewise affine permutations of $X$ is denoted by $\widehat{\PAff^{\bowtie}}$. We also consider for each of these groups the subgroup composed of all piecewise orientation-preserving elements by replacing the symbol $``\bowtie"$ by the symbol $``+"$.\\
Let us observe that when $G \subset \PC^+$ Corollary \ref{Corollary vanishing of the Kapoudjian class} is trivial. Indeed, in this case $G$ can be lifted inside $\widehat{\PC^+}$ itself. However, such a lift does not exist for $\PC^{\bowtie}$ or even $\IET^{\bowtie}$, as was proved in \cite{Cornulier_rea}.\\

The idea of proof of Theorem \ref{Theorem Existence of a nontrivial group homomorphism onto Z/2Z} is to associate for every $f \in \widehat{\PC^{\bowtie}}$ and every finite partition $\cP$ of $\mathopen{[} 0,1 \mathclose{[}$ into intervals associated with $f$, two numbers. The first is the number of interval of $\cP$ where $f$ is order-reversing and the second is the signature of a particular finitely supported permutation. The next step is to prove that the sum modulo $2$ of this two numbers is independent from the choice of partition. Then we show that it is enough to prove that $\eps|_{\IET^{\bowtie}}$ is a group homomorphism. For this we show that it is additive when we look at the composition of two elements of $\widehat{\IET^{\bowtie}}$ by calculate the value of the signature with a particular partition.\\

In Section \ref{Section Normal subgroups of widehat PC^ bowtie and some subgroups}, we apply these results to the study of normal subgroups of $\widehat{\PC^{\bowtie}}$ and certain subgroups. More specifically we prove:

\begin{Thm}\label{Theorem Normal subgroups of some subgroups of widehat PC bowtie}
Let $\widehat{G}$ be a subgroup of $\widehat{\PC^{\bowtie}}$ containing $\mfS_{\mathrm{fin}}$ and such that its projection $G$ in $\PC^{\bowtie}$ is simple nonabelian. Then $\widehat{G}$ has exactly five normal subgroups given by the list: $\lbrace \lbrace 1 \rbrace , \mfA_{\mathrm{fin}}, \mfS_{\mathrm{fin}} , \Ker(\eps) , \widehat{G} \rbrace$.
\end{Thm}

We denote by $\widehat{\IET^+_{\text{rc}}}$ the subgroup of $\widehat{\IET^+}$ composed of all right-continuous elements. We know that it is naturally isomorphic to $\IET^+$. The same is true when we replace $\IET^+$ by $\PAff^+$ or $\PC^+$. This allows us to use the work of P. Arnoux \cite{ArnouxThese} and the one of N. Guelman and I. Liousse \cite{GuelmanLiousse2019} where they prove that $\IET^{\bowtie}$, $\PC^+$ and $\PAff^+$ are simple. From this we deduce:

\begin{Thm}
The groups $\PC^{\bowtie}$ and $\PAff^{\bowtie}$ are simple.
\end{Thm}

This gives us some examples of groups that satisfy the conditions of the Theorem \ref{Theorem Normal subgroups of some subgroups of widehat PC bowtie}.
~\\

Finally Section \ref{Section About some Normalizers} is independent and we study some normalizers and in particular we show that the behaviour when we look the group inside $\widehat{\PC^{\bowtie}}$ or $\PC^{\bowtie}$ may not be the same. We denote by $\cR \in \IET^{\bowtie}$ the map $x \mapsto 1-x$. Then we define $\IET^-$ as the coset $\cR . \IET^+$ and $\PC^-$ as the coset $\cR . \PC^+$. Then the groups $\IET^{\pm}:= \IET^+ \cup \IET^+$ and $\PC^{\pm}:= \PC^+ \cup \PC^-$ are well-defined.

\begin{Prop}
The subgroup $\widehat{\IET^+_{\mathrm{rc}}}$ (resp.\ $\widehat{\PC^+_{\mathrm{rc}}}$) is its own normalizer in $\widehat{\IET^{\bowtie}}$ (resp.\ $\widehat{\PC^+_{\mathrm{rc}}}$).
The normalizer of $\IET^+$ ($\PC^+$ respectively) in $\IET^{\bowtie}$ ($\PC^{\bowtie}$ respectively) is $\IET^{\pm}$ ($\PC^{\pm}$ respectively).
\end{Prop}

\medskip

\noindent \textbf{Acknowledgments.} I would like to thank Y. Cornulier, P. de la Harpe and N. Matte Bon for corrections, remarks and discussions on preliminary versions of this paper.

\section{Preliminaries}

For every real interval $I$ we denote by $I^{\circ}$ its interior in $\mathbb{R}$ and if $I=\mathopen{[}0,t \mathclose{[}$ we agree that its interior is $\mathopen{]}0,t \mathclose{[}$.

\subsection{Partitions associated}~\

An important tool to study elements in $\widehat{\PC^{\bowtie}}$ and $\PC^{\bowtie}$ are partitions into intervals of $\mathopen{[}0,1 \mathclose{[}$. All partitions are assumed to be finite.

\begin{Def}\label{Definition partition associated}
For every $f$ in $\widehat{\PC^{\bowtie}}$, a finite partition $\cP$ into right-open and left-closed intervals of $\mathopen{[}0,1 \mathclose{[}$ is called \textit{a partition into intervals associated with $f$} if and only if $f$ is continuous on the interior of every interval of $\cP$. We denote by $\Pi_f$ the set of all partitions into intervals associated with $f$.\\
We define also \textit{the arrival partition of $f$ associated with $\cP$}, denoted $f(\cP)$, the partition of $\mathopen{[}0,1\mathclose{[}$ composed of all right-open and left-closed intervals such that their interior is equal to the image by $f$ of the interior of an interval of $\cP$.
\end{Def}
 
\begin{Rem}
For every $f$ in $\widehat{\PC^{\bowtie}}$ there exists a unique partition $\cP_f^{\min}$ associated with $f$ which has a minimal number of intervals. It is actually minimal in the sense of refinement: $\Pi_f$ consists precisely of the set of partitions refining $\mathcal{P}^{\min}_f$.
\end{Rem}

\subsection{Decompositions}~\

We define a family of elements which plays an important role inside our groups:

\begin{Def}
Let $I$ be a non-empty right-open and left-closed subinterval of $\mathopen{[} 0,1 \mathclose{[}$. The element $f \in \widehat{\PC^{\bowtie}}$ which sends the interior of $I$ on itself with slope $-1$ while fixing the rest of $\mathopen{[} 0,1 \mathclose{[}$ is called the \textit{I-flip}. We define \textit{a flip} as any $I$-flip for some $I$.
\end{Def}

From the definition we deduce a decomposition inside $\widehat{\IET^{\bowtie}}$ and $\widehat{\PC^{\bowtie}}$. 

\begin{Prop}\label{Proposition Decomposition in widehat(IET^bowtie)}
Let $h$ be an element of $\widehat{\IET^{\bowtie}}$. There exist $f,g \in \widehat{\IET^{+}_{\mathrm{rc}}}$ and $r,s$ finite products of flips and $\sigma,\tau$ finitely supported permutations such that $h=r\sigma f=g \tau s$.
\end{Prop}

\begin{proof}
Let $h$ be an element of $\widehat{\IET^{\bowtie}}$, $n \in \NN$ and $\cP:= \lbrace I_1,I_2, \ldots ,I_n\rbrace \in \Pi_h$ (\S~\ref{Definition partition associated}). We denote by $h(\cP):= \lbrace J_1,J_2, \ldots , J_n \rbrace$ the arrival partition of $h$ associated with $\cP$. Let $g$ be the map that sends $I_j^{\circ}$ on $J_j^{\circ}$ by preserving the order and acts as $h$ for every left endpoints of $I_j$ for every $1 \leq j \leq n$. Note that $g$ is bijective and then belongs to $\widehat{\IET^+}$. For $1 \leq j \leq n$ let $r_j$ be the $J_j$-flip if $h$ is order-reversing on $I_j$ otherwise let $r_j$ be the identity. Let $r$ be the product of all $r_j$, we can notice that $r$ fixes all endpoints of $J_j$ for every $1 \leq j \leq n$. Then it is just a verification to check that $h=rg$. Now as $g$ belongs to $\widehat{\IET^{+}}$ there exists $\sigma$ in $\mfS_n$ such that $g=\sigma f$ with $f$ in $\widehat{\IET^+_{\text{rc}}}$.\\
The other decomposition follows by decomposing $h^{-1}$ under the previous decomposition.
\end{proof}

\begin{Prop}\label{Proposition Decomposition in widehat(PC^bowtie)}
For every $h$ in $\widehat{\PC^{\bowtie}}$ there exist $\phi$ and $\psi$ two order-preserving homeomorphisms of $\mathopen{[} 0,1 \mathclose{[}$ and $f,g$ in $\widehat{\IET^{\bowtie}}$ such that $h=\psi \circ f= g \circ \phi$.
\end{Prop}

\begin{proof}
Let $\lambda$ be the Lebesgue measure on $\mathopen{[}0,1 \mathclose{[}$. Let $h \in \widehat{\PC^{\bowtie}}$ and $\cP \in \Pi_h$. Then there exist $\phi,\psi \in \Homeo^+(\mathopen{[} 0,1 \mathclose{[})$ such that for every $I \in \cP$, $\lambda(\phi(I))=\lambda(h(I))$ and $\lambda(\psi(h(I)))=\lambda(I)$. Then $h \circ \phi$ and $\psi \circ h $ belongs to $\widehat{\IET^{\bowtie}}$.
\end{proof}

\section{Construction of the signature homomorphism}\label{Section The signature homomorphism}

In our case we have $X= \mathopen{[}0,1 \mathclose{[}$ and $\widehat{\PC^{\bowtie}}$ is a subgroup of $\mfS(X)$. We denote here $\mfS_{\mathrm{fin}}=\mfS_{\mathrm{fin}}(X)$ and $\eps_{\mathrm{fin}}$ the classical signature on $\mfS_{\mathrm{fin}}$ taking values in $(\faktor{\ZZ}{2\ZZ},+)$.

\subsection{Definitions}

\begin{Def}
Let $h$ be an element of $ \widehat{\PC^{\bowtie}}$, $n \in \NN$ and $\cP=\lbrace I_1,I_2, \ldots ,I_n \rbrace \in \Pi_h$. For every $1 \leq j \leq n$, let $\alpha_j$ be the left endpoint of $I_j$ and $\beta_j$ be the left endpoint of $h(I_j^{\circ})$. We define the \textit{default of pseudo right continuity for $h$ about $\cP$} denoted $\sigma_{(h,\cP)}$ as the finitely supported permutation which sends $h(\alpha_j)$ to $\beta_j$ for every $1 \leq j \leq n$ (this is well-defined because the set of all $h(\alpha_j)$ is equal to the set of all $\beta_j$).
\end{Def}

\begin{Def}
Let $h$ be an element of $ \widehat{\PC^{\bowtie}}$ and $\cP \in \Pi_h$. Let $k$ be the number of interval of $\cP$ on which $h$ is order-reversing. We called the \textit{flip number of $h$ about $\cP$} the number $k$. We denote it by $R(h,\cP)$.
\end{Def}

\begin{Def}
For $h \in \widehat{\PC^{\bowtie}}$ and $\cP \in \Pi_h$, define $\eps(h,\cP) \in \faktor{\ZZ}{2\ZZ}$ as $R(h,\cP) + \eps_{\mathrm{fin}}(\sigma_{(h,\cP)}) ~[\mathrm{mod}\ 2]$. We define also $\eps(h)=\eps(h,\cP_h^{\mathrm{fin}})$.
\end{Def}

\begin{Prop}
For every $\tau \in \mfS_{\mathrm{fin}}$ and every $\cP \in \Pi_{\tau}$ we have $\eps(\tau,\cP)=\eps_{\mathrm{fin}}(\tau)$.
\end{Prop}

\begin{proof}
It is clear that for every $\tau \in \mfS_{\mathrm{fin}}$ and every partition $\cP$ associated with $\tau$ we have $R(\tau,\cP)=0$ and $\sigma_{(\tau,\cP)}=\tau$.
\end{proof}

We deduce that $\eps$ extends the classical signature $\eps_{\mathrm{fin}}$. Thus we will write $\eps$ instead of $\eps_{\mathrm{fin}}$.

\begin{Prop}\label{Proposition widehat PC^+ rc is in the kernel of epsilon}
Every right-continuous element $f$ of $\widehat{\PC^+}$ satisfies $\eps(f,\cP)=0$ for every $\cP \in \Pi_f$.
\end{Prop}

\begin{proof}
In this case, for every partition $\cP$ into intervals associated with $f$ we always have $R(f,\cP)=0$ and $\sigma_{(f,\cP)}=\Id$.
\end{proof}

\subsection{Proof of Theorem \ref{Theorem Existence of a nontrivial group homomorphism onto Z/2Z}}~\

In order to prove that $\eps$ is a group homomorphism, it is useful to calculate $\eps(h)$ thanks to $\eps(h,\cP)$ for every $h \in \widehat{\PC^{\bowtie}}$ and $\cP \in \Pi_h$.

\begin{Lem}\label{Lemma the signature is independent of the partition}
For every $h \in \widehat{\PC^{\bowtie}}$ and every $\cP \in \Pi_h$ we have $\eps(h)=\eps(h,\cP)$.
\end{Lem}

\begin{proof}
Let $h$ and $\cP$ be as in the statement. By minimality of $\cP_h^{\min}$, in term of refinement, we deduce that there exist $n \in \NN$ and $\cP_1,\cP_2,\ldots ,\cP_n \in \Pi_h$ such that:
\begin{enumerate}
\item $\cP_1=\cP_h^{\min}$;
\item $\cP_n=\cP$;
\item for every $2 \leq i \leq n$ the partition $\cP_i$ is a refinement of the partition $\cP_{i-1}$ where only one interval of $\cP_{i-1}$ is cut into two.
\end{enumerate}

Hence it is enough to show $\eps(h,\cQ)=\eps(h,\cQ')$ where $\cQ, \cQ' \in \Pi_h$ such that there exist consecutive intervals $I,J \in \cQ$ with $I\cup J \in \cQ'$ and $\cQ' \smallsetminus{\lbrace I \cup J \rbrace}=\cQ \smallsetminus{\lbrace I,J \rbrace}$.

Let $\alpha$ be the left endpoint of $I$ and let $x$ be the right endpoint of $I$ ($x$ is also the left endpoint of $J$). There are only two cases but in both cases, we know that $\sigma_{(h,\cQ)}=\sigma_{(h,\cQ')}$ except maybe on $h(\alpha)$ and $h(x)$:
\begin{enumerate}
\item The first case is when $h$ is order-preserving on $I \cup J$. Then as $\cQ \setm \lbrace I,J \rbrace = \cQ' \setm \lbrace I \cup J \rbrace$ we get $R(h,\cQ)=R(h,\cQ')$. As $h$ is order-preserving on the interior of $I \cup J$ we know that $\sigma_{(h,\cQ')}(h(\alpha))$ is the left endpoint of $h(I\cup J)$ which is the left endpoint of $h(I)$ thus equals to $\sigma_{(h,\cQ)}(h(\alpha))$. With the same reasoning we deduce that $\sigma_{(h,\cQ')}(h(x))=\sigma_{(h,\cQ)}(h(x))$ hence $\sigma_{(h,\cQ)}=\sigma_{(h,\cQ')}$. Thus in $\faktor{\ZZ}{2\ZZ}$ we have $R(h,\cQ') + \eps(\sigma_{(h,\cQ')})=R(h,\cQ) + \eps(\sigma_{(h,\cQ)})$.
\item The second case is when $h$ is order-reversing on $I \cup J$. Then we get $R(h,\cQ)=R(h,\cQ')+1$. This time $\sigma_{(h,\cQ')}(h(\alpha))$ is still the left endpoint of $h(I \cup J)$ which is the left endpoint of $h(J)$ thus equals to $\sigma_{(h,\cQ)}(h(x))$. With the same reasoning we deduce that $\sigma_{(h,\cQ')}(h(x))=\sigma_{(h,\cQ)}(h(\alpha))$. Then by denoting $\tau$ the transposition $(h(x) ~ \sigma_{(h,\cQ')}(h(\alpha)))$, we obtain $\sigma_{(h,\cQ)}=\tau \circ \sigma_{(h,\cQ')}$. We must notice that the transposition is not the identity because $h^{-1}(\sigma_{(h,\cQ')}(h(\alpha)))$ is an endpoint of one of the intervals of $\cQ'$ and $x$ is not.\\
In conclusion in $\faktor{\ZZ}{2\ZZ}$ we have:
\[ R(h,\cQ') + \eps(\sigma_{(h,\cQ')})=R(h,\cQ') +1 +1+ \eps(\sigma_{(h,\cQ')})=R(h,\cQ) + \eps(\sigma_{(h,\cQ)}) \]
\end{enumerate}

\begin{center}
\includegraphics[scale=0.60]{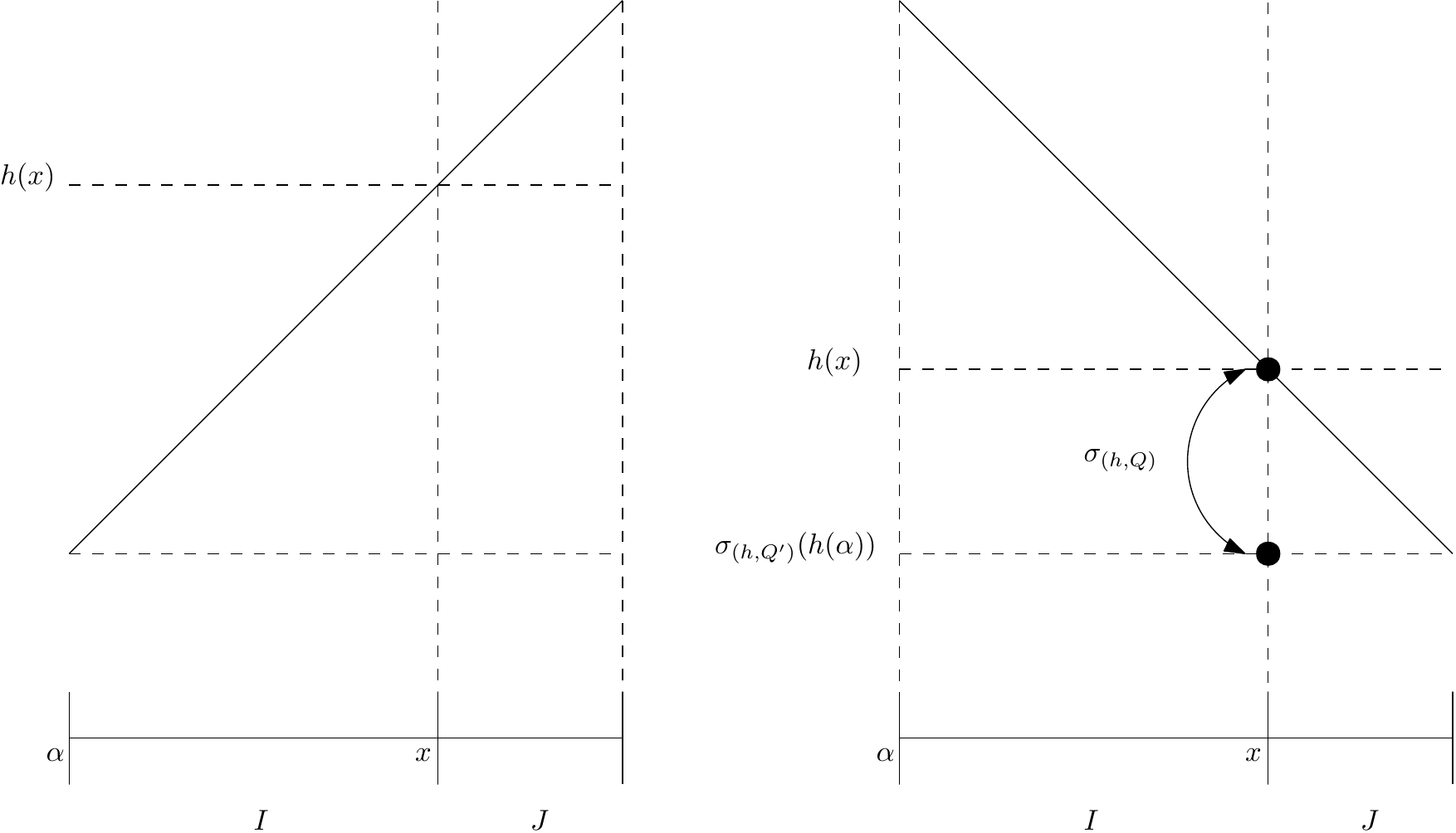}
\captionof{figure}{\footnotesize{Illustrations of the two cases appearing in Lemma \ref{Lemma the signature is independent of the partition}.\\ On the left we assume $h$ order-preserving on $I \cup J$ and see that $\sigma_{(h,\cQ)} (h(x))= \sigma_{(h,\cQ')}(h(x))$. On the right we assume $h$ order-reversing on $I \cup J$ and see that $\sigma_{(h,\cQ)} (h(x)) = (h(x) ~ \sigma_{(h,\cQ')}(h(\alpha)) ) \circ \sigma_{(h,\cQ')}(h(x))$.}}
\end{center}

\end{proof}

If $\phi \in \Homeo^+(\mathopen{[}0,1 \mathclose{[})$, then it follows from Proposition \ref{Proposition widehat PC^+ rc is in the kernel of epsilon} that $\eps(\phi)=0$. We improve this, showing that $\eps$ is invariant by the action of $\Homeo^+(\mathopen{[} 0,1 \mathclose{[})$ on $\widehat{\PC^{\bowtie}}$. 

\begin{Lem}
For every $h \in \widehat{\PC^{\bowtie}}$ and every $\phi \in \Homeo^+(\mathopen{[} 0,1 \mathclose{[})$ we have $\eps(h \phi)=\eps(h)=\eps(\phi h)$.
\end{Lem}

\begin{proof}
Let $h \in \widehat{\PC^{\bowtie}}$ and $\phi \in \Homeo^+(\mathopen{[} 0,1 \mathclose{[})$ be as in the statement. Let $n \in \NN$ and $\cP:= \lbrace I_1,I_2, \ldots ,I_n \rbrace \in \Pi_h$. Then $\cQ:=\lbrace \phi^{-1}(I_1),\phi^{-1}(I_2), \ldots, \phi^{-1}(I_n) \rbrace$ is in $\Pi_{h\phi}$. We know that $\phi$ is order preserving then for every $1 ,\leq i \leq n$, $h \phi$ preserves (reverses respectively) the order on $\phi^{-1}(I_i)$ if and only if $h$ preserves (reverses respectively) the order on $I_i$, so $R(h,\cP)=R(h\phi,\cQ)$. We can notice that the left endpoint of $\phi^{-1}(I_i)$ (denoted by $\alpha_i$) is send on the left endpoint of $I_i$ (denoted by $a_i$) by $\phi$ hence $h(a_i)=h\phi(\alpha_i)$ has to be send on $\sigma_{(h,\cP)}(h(a_i))$ so $\sigma_{(h\phi,\cQ)}=\sigma_{(h,\cP)}$. we deduce that $\eps(h\phi)= \eps(h)$.\\
The other equality has a similar proof. We denote $h(\cP)$ the arrival partition of $h$ associated with $\cP$. We know that $\phi$ is continuous thus $h(\cP)$ is in $\Pi_{\phi}$ and we deduce that $\cP \in \Pi_{\phi h}$. Also $\phi$ is order-preserving then $R(h,\cP)=R(\phi h, \cP))$. We know that $\sigma_{(\phi,h(\cP))}=\Id$ then we can notice that $\phi \circ \sigma_{(h,\cP)} \circ h$ send the left endpoint of $I_i$ to the left endpoint of $\phi h (I_i^{\circ})$. Then $\sigma_{(\phi h, \cP)}=\phi \sigma_{(h,\cP)} \phi^{-1}$ and we deduce that $\eps(\sigma_{(\phi h, \cP)})=\eps(\sigma_{(h,\cP)})$. Hence $\eps(\phi h)=\eps(h)$.
\end{proof}

Thanks to Proposition \ref{Proposition Decomposition in widehat(PC^bowtie)} it is enough to prove that $\eps |_{\widehat{\IET^{\bowtie}}}$ is a group homomorphism.

\begin{Lem}
The map $\eps |_{\widehat{\IET^{\bowtie}}}$ is a group homomorphism.
\end{Lem}

\begin{proof}
Let $f,g \in \widehat{\IET^{\bowtie}}$. Let $\cP \in \Pi_f$ and $\cQ \in \Pi_g$. For every $I \in \cQ$ (resp.\ $J \in \cP$) we denote by $\alpha_I$ (resp.\ $\beta_J$) the left endpoint of $I$ (resp.\ $J$).  Up to refine $\cP$ and $\cQ$ we can assume that $\cP=g(\cQ)$ thus $g(\lbrace \alpha_I \rbrace_{I \in \cQ})=\lbrace \beta_J \rbrace_{J \in \cP}$. Then $Q \in \Pi_{f \circ g}$ and for every $K \in f \circ g(Q)$ we denote by $\gamma_K$ the left endpoint of $K$.\\
In $\faktor{\ZZ}{2\ZZ}$, we get immediately that $R(f \circ g, Q)=R(g,Q) + R(f,g(Q))$. Now we want to describe the default of pseudo right continuity for $f \circ g$ about $\cQ$. We recall that $\sigma_{(f \circ g,\cQ)}$ is the permutation that sends $f \circ g (\alpha_I)$ on $\gamma_{f \circ g (I)}$ for every $I \in \cQ$ while fixing the rest of $\mathopen{[}0,1 \mathclose{[}$. Furthermore $\sigma_{(g,\cQ)}(g(\alpha_I))=\beta_{g(I)}$ and $\sigma_{(f,g(\cQ))}(f(\beta_{g(I)}))=\gamma_{f \circ g(I)}$. Then $\sigma_{(f,g(\cQ))} \circ f \circ \sigma_{(g,\cQ)} \circ g (\alpha_I)=\gamma_{f \circ g(I)}$ and we deduce that the permutation $\sigma_{(f,g(\cQ))} \circ f \circ \sigma_{(g,\cQ)} \circ f^{-1}$ sends $f \circ g (\alpha_I)$ on $\gamma_{f \circ g (I)}$ for every $I \in \cQ$ while fixing the rest of $\mathopen{[}0 , 1 \mathclose{[}$. Thus $\sigma_{(f\circ g, \cQ)}=\sigma_{f,g(\cQ)} \circ f \circ \sigma_{(g,\cQ)} \circ f^{-1}$. Then $\eps(\sigma_{(f\circ g, \cQ)})=\eps(\sigma_{f,g(\cQ)}) + \eps(\sigma_{(g,\cQ)})$ and we conclude that $\eps(f \circ g)= \eps(f) + \eps(g)$.
\end{proof}

\begin{Coro}
The map $\eps$ is a group homomorphism. \qed
\end{Coro}

\section{Normal subgroups of $\widehat{\PC^{\bowtie}}$ and some subgroups}\label{Section Normal subgroups of widehat PC^ bowtie and some subgroups}

Here we present some corollaries of Theorem \ref{Theorem Existence of a nontrivial group homomorphism onto Z/2Z}. For every group $G$ we denote by $D(G)$ its derived subgroup.

\begin{Def}
For every group $H$, we define $J_3(H)$ as the subgroup generated by elements of order 3.
\end{Def}

Let $\widehat{G}$ be a subgroup of $\widehat{\PC^{\bowtie}}$ containing $\mfS_{\mathrm{fin}}$. We denote by $G$ its projection on $\PC^{\bowtie}$. We recall that $\mfA_{\mathrm{fin}}$ is a normal subgroup of $\widehat{G}$, and has a trivial centraliser. We deduce that for every nontrivial normal subgroup $H$ of $\widehat{G}$ contains $\mfA_{\mathrm{fin}}$.

From the short exact sequence:
$$1 \longrightarrow \mfS_{\mathrm{fin}} \longrightarrow \widehat{G} \longrightarrow G \longrightarrow 1
$$
we deduce the next short exact sequence which is a central extension:
$$
1 \longrightarrow \faktor{\ZZ}{2 \ZZ} \longrightarrow \faktor{\widehat{G}}{\mfA_{\mathrm{fin}}} \longrightarrow G \longrightarrow 1.
$$

This short exact sequence splits because the signature $\eps_{|\widehat{G}}: \widehat{G} \rightarrow \faktor{\ZZ}{2\ZZ}$ constructed in \S~\ref{Section The signature homomorphism} is a retraction. Then we deduce that $\faktor{\widehat{G}}{\mfA_{\mathrm{fin}}}$ is isomorphic to the direct product $\faktor{\ZZ}{2\ZZ} \times G$.

\begin{Coro}
The projection $\widehat{G}_{\mathrm{ab}} \rightarrow G_{\mathrm{ab}}$ extends in an isomorphism $\widehat{G}_{\mathrm{ab}} \sim G_{\mathrm{ab}} \times \faktor{\ZZ}{2\ZZ}$. Furthermore $D(\widehat{G})=\Ker(\eps) \cap \widehat{D(G)}$ is a subgroup of index $2$ in $\widehat{D(G)}$. In particular, if $G$ is a perfect group then $\widehat{G}_{\mathrm{ab}}= \faktor{\ZZ}{2\ZZ}$.
\end{Coro}

\begin{Coro}
Let $\widehat{G}$ be a subgroup of $\widehat{\PC^{\bowtie}}$ containing $\mfS_{\mathrm{fin}}$ and such that its projection $G$ in $\PC^{\bowtie}$ is simple nonabelian. Then $\widehat{G}$ has exactly $5$ normal subgroups given by the list: $\lbrace \lbrace 1 \rbrace , \mfA_{\mathrm{fin}}, \mfS_{\mathrm{fin}} , \Ker(\eps) , \widehat{G} \rbrace$.
\end{Coro}

\begin{proof}
Let $\widehat{G}$ as in the statement. First we immediately check that the subgroups in the list are distinct normal subgroups of $\widehat{G}$. In the case of $\Ker(\eps)$, there exists $g \in \widehat{G} \setm \mfS_{\mathrm{fin}}$ thus either $g \in \Ker(\eps)\setm \mfS_{\mathrm{fin}}$ or $\sigma g \in \Ker(\eps) \setm \mfS_{\mathrm{fin}}$ for any transposition $\sigma$.\\
Second let $H$ be a normal subgroup of $\widehat{G}$ distinct from $\lbrace 1 \rbrace$. Then it contains $\mfA_{\mathrm{fin}}$. Also $\faktor{H}{\mfA_{\mathrm{fin}}}$ is a normal subgroup of $\faktor{\widehat{G}}{\mfA_{\mathrm{fin}}} \simeq \faktor{\ZZ}{2\ZZ} \times G$. Furthermore $G$ is simple then there are only four possibilities for $\faktor{H}{\mfA_{\mathrm{fin}}}$. As two normal subgroups $H,K$ of $\widehat{G}$ containing $\mfA_{\mathrm{fin}}$ such that $\faktor{H}{\mfA_{\mathrm{fin}}}=\faktor{K}{\mfA_{\mathrm{fin}}}$ are equal, we deduce that $\widehat{G}$ has at most $5$ normal subgroups.
\end{proof}

\begin{Coro}\label{Corollary When D(G)=Ker(epsilon)?}
Let $\widehat{G}$ be a subgroup of $\widehat{\PC^{\bowtie}}$ containing $\mfS_{\mathrm{fin}}$ and such that its projection $G$ in $\PC^{\bowtie}$ is simple nonabelian. If there exists an element of order $3$ in $G \setm \mfA_{\mathrm{fin}}$ then $J_3(\widehat{G})=\Ker(\eps)=D(\widehat{G})$. \qed
\end{Coro}

\begin{Rem}
In the context of topological-full groups, the group $J_3(G)$ appears naturally (with some mild assumptions) and is denoted by $\mathsf{A}(G)$ by Nekrashevych in \cite{nekrashevych_2019}. In some case of topological-full groups of minimal groupoids (see \cite{matui2015topological}) we have the equality $A(G)=D(G)$ thanks to the simplicity of $D(G)$. In spite of the analogy, it is not clear that the corollary can be obtained as particular case of this result.
\end{Rem}

\begin{Rem}
A lot of groups satisfy the conditions of Corollary \ref{Corollary When D(G)=Ker(epsilon)?}. When $\widehat{G}$ contains $\widehat{\IET^+}$ there is an element of order $3$ in $G \setm \mfA_{\mathrm{fin}}$. We recall that $\IET^{\bowtie}$, $\PC^+$ and $\PAff^+$ are simple (see \cite{ArnouxThese, GuelmanLiousse2019}). Thus these groups satisfy the conditions of Corollary \ref{Corollary When D(G)=Ker(epsilon)?}. The next theorem add $\PC^{\bowtie}$ and $\PAff^{\bowtie}$ to the list of examples.
\end{Rem}

\begin{Thm}\label{Theorem PC an Paff bowtie are simple}
The groups $\PC^{\bowtie}$ and $\PAff^{\bowtie}$ are simple.
\end{Thm}

\begin{Lem}\label{Lemma generators for widehat IET bowtie and IET bowtie}
The group $\IET^{\bowtie}$ is generated by flips (=images of flips from $\widehat{\IET^{\bowtie}}$).
\end{Lem}

\begin{proof}
By Proposition \ref{Proposition Decomposition in widehat(IET^bowtie)} it is enough to show that $\IET^+$ is generated by flips.\\
For every consecutive, right-open and left-closed subintervals $I$ and $J$ of $\mathopen{[}0,1\mathclose{[}$, we define $R_{I,J}$ the map that exchanges $I$ and $J$. They are elements of $\widehat{\IET^+_{\mathrm{rc}}}$ and they formed a generating set. Then their image $r_{I,J}$ in $\IET^{\bowtie}$ is a generating set of $\IET^{+}$. For every right-open and left-closed subinterval $I$ of $\mathopen{[}0,1\mathclose{[}$, we define $s_I$ the $I$-flip. Take $I$ and $J$ be two consecutive, right-open and left-closed subintervals of $\mathopen{|}0,1\mathclose{[}$. Then $r_{I,J}=s_Is_Js_{I\cup J}$.
\end{proof}

\begin{proof}[Proof of Theorem \ref{Theorem PC an Paff bowtie are simple} (sketched)]~\\
Since the argument in \cite{ArnouxThese} could also be adapted, we only provide a sketch.\\
We work with elements of $\PC^{\bowtie}$; all intervals below are meant modulo finite subsets. Let $N$ be a nontrivial normal subgroup of $\PC^{\bowtie}$ (resp.\ $\PAff^{\bowtie}$). Let $g$ be a nontrivial element of $N$. There exists a subinterval $I$ of $\mathopen{[}0,1\mathclose{[}$ such that:
\begin{enumerate}
\item $g$ is continuous (resp.\ affine) on $I$,
\item $g(I) \cap I = \emptyset$ (modulo finite subsets),
\item $I\cup g(I)\neq [0,1[$ (modulo finite subsets).
\end{enumerate}
Let $f$ be the $I$-flip. If $g$ is affine on $I$ then $h=gfg^{-1}f^{-1}$ is the product of the $I$-flip with the $g(I)$-flip. Observe that $h$ is conjugate to a single flip by a suitable element of $\IET^+$. If $g$ is only continuous then $h$ is still of order $2$ and it is conjugate in $\PC^{\bowtie}$ to a single flip. Conjugating by elements of $\PAff^+$, one obtains that $N$ contains flips of intervals of all possible lengths, and hence contains all flips. Thanks to Lemma \ref{Lemma generators for widehat IET bowtie and IET bowtie} we know that $\IET^{\bowtie}$ is generated by the set of flips thus $N$ contains $\IET^{\bowtie}$, in particular $N$ intersects $\PC^+$ (resp.\ $\PAff^+$) nontrivially. By simplicity of $\PC^+$ (resp.\ $\PAff^+$) we deduce that $N$ contains $\PC^{\bowtie}=\langle \PC^+,\IET^{\bowtie}\rangle$ (resp.\ $\PAff^{\bowtie}=\langle \PAff^+,\IET^{\bowtie} \rangle$).
\end{proof}

\section{About some Normalizers}\label{Section About some Normalizers}

Here we show that computing normalizers inside $\widehat{\PC^{\bowtie}}$ and $\PC^{\bowtie}$ may leads to different behaviour. We look the case of $\PC^+$, $\IET^+$ and $\widehat{\PC^+_{\mathrm{rc}}}$ and $\widehat{\IET^+_{\mathrm{rc}}}$.

\begin{Prop}\label{Proposition Normalizer of IET^+}
The normalizer of $\IET^+$ in $\IET^{\bowtie}$ is reduced to $\IET^{\pm}$.
\end{Prop}

\begin{proof}
Let $f \in \IET^{+}$ and $g \in \IET^{\pm}$. If $g \in \IET^+$ then $gfg^{-1} \in \IET^+$. We assume $g \in \IET^-$ then $gfg^{-1}=(g \circ \cR) \circ (\cR \circ f \circ \cR) \circ (\cR \circ g) \in \IET^+$.\\
For the inclusion from left to right, let $g \in \IET^{\bowtie} \setm \IET^{\pm}$ and let $\widehat{g}$ be a representative of $g$ in $\widehat{\IET^{\bowtie}}$. Hence we can find $I,J,K,L$ four right-open and left-closed intervals of the same length such that their image by $\widehat{g}$ are intervals and such that $\widehat{g}$ is order-reversing on $I$ and order-preserving on $J,K$ and $L$.
We define $\widehat{f} \in \widehat{\IET^+}$ as the element which exchanges $\widehat{g}(I)$ with $\widehat{g}(J)$ and $\widehat{g}(K)$ with $\widehat{g}(L)$ while fixing the rest of $\mathopen{[}0,1 \mathclose{[}$. Then the image $f$ of $\widehat{f}$ in $\IET^{+}$ is not trivial and $\widehat{g}\widehat{f}\widehat{g}^{-1} \notin \widehat{\IET^+}$ implies $gfg^{-1} \notin \IET^+$.
\end{proof}
~\\

A similar argument stands for the case of $\PC$ thus we obtain:

\begin{Prop}
The normalizer of $\PC^+$ in $\PC^{\bowtie}$ is reduced to $\PC^{\pm}$. \qed
\end{Prop}

We now take a look to inside $\widehat{\PC^{\bowtie}}$:

\begin{Prop}
The normalizer of $\widehat{\IET^+_{\mathrm{rc}}}$ in $\widehat{\IET^{\bowtie}}$ is $\widehat{\IET^+_{\mathrm{rc}}}$.
\end{Prop}

\begin{proof}
Let $g$ be an element of $\widehat{\IET^{\bowtie}}$ which is not the identity. There are two cases:
\begin{enumerate}
\item If $g \in \widehat{\IET^+} \setm \widehat{\IET^+_{\mathrm{rc}}}$ then $g=\sigma g'$ with $\sigma \in \mfS_{\mathrm{fin}} \setm \lbrace \Id \rbrace$ and $g' \in \widehat{\IET^+_{\mathrm{rc}}}$. Then for every $f \in \widehat{\IET^+_{\mathrm{rc}}}$ we have $gfg^{-1}=\sigma g'fg'^{-1} \sigma^{-1}$. Thus it is enough to treat the case of $\mfS_{\mathrm{fin}}$. Let us assume $g \in \mfS_{\mathrm{fin}}$ then let $x$ in the support of $g$. There exist two consecutive right-open and left-closed intervals $I$ and $J$ of the same length such that $x$ is the right endpoint of $I$ (and the left endpoint of $J$). Up to reduce $I$ and $J$ we can assume that $I$ does not intersect the support of $g$. Then let $f \in \widehat{\IET^+_{\mathrm{rc}}}$ which exchanges $I$ and $J$ while fixing the rest of $\mathopen{[}0,1 \mathclose{[}$. Then $gfg^{-1}$ exchanges the interior of $I$ with the interior of $J$ but $gfg^{-1}(x)$ is not equal to $f(x)$ because $f(x)$ is the left endpoint of $I$ and $I$ does not intersect the support of $g$. Then we deduce that $gfg^{-1}$ is not right-continuous on $J$.
\item If $g \in \widehat{\IET^{\bowtie}} \setm \widehat{\IET^+}$. Then we can find two consecutive subinterval $I$ and $J$ where $g$ is continuous and order-reversing on $I \cup J$. Let $a$ be the right endpoint of $J$. Let $f$ be the element in $\widehat{\IET^+_{\text{rc}}}$ which exchanges $I$ and $J$. Then $gfg^{-1}$ exchanges the interior of $g(J)$ with the interior of $g(I)$. However the left endpoint of $g(J)$ is send by $g^{-1}$ on $a$ which is fixed by $f$. Then $gfg^{-1}$ fixes the left endpoint of $g(J)$, thus $gfg^{-1}$ is not right-continuous on $g(J)$.
\end{enumerate}
\end{proof}

A similar argument stands for the case of $\PC$ thus we obtain:

\begin{Prop}
The normalizer of $\widehat{\PC^+_{\mathrm{rc}}}$ in $\widehat{\PC^{\bowtie}}$ is $\widehat{\PC^+_{\mathrm{rc}}}$. \qed
\end{Prop}

\nocite{Arnoux}
\bibliographystyle{abbrv}
\bibliography{biblio}

\begin{thebibliography}{10}

\bibitem{Arnoux}
P.~Arnoux.
\newblock {\'E}changes d’intervalles et flots sur les surfaces.
\newblock {\em Ergodic theory (Sem., Les Plans-sur-Bex, 1980)}, pages 5--38,
  1981.

\bibitem{ArnouxThese}
P.~Arnoux.
\newblock {\em Un invariant pour les echanges d'intervalles et les flots sur
  les surfaces}.
\newblock PhD thesis, Université de Reims, 1981.

\bibitem{2019arXiv190105065C}
Y.~Cornulier.
\newblock Near actions.
\newblock {\em ArXiv:1901.05065}, 2019.

\bibitem{Cornulier_rea}
Y.~Cornulier.
\newblock Realizations of groups of piecewise continuous transformations of the
  circle.
\newblock {\em ArXiv:1902.06991}, 2019.
\newblock To appear in J. Modern Dyn.

\bibitem{GuelmanLiousse2019}
N.~Guelman and I.~Liousse.
\newblock Bounded simplicity of affine interval exchange transformations and
  interval exchange transformations.
\newblock {\em ArXiv:1910.08923}, 2019.

\bibitem{Kapoudjian2002}
C.~Kapoudjian.
\newblock Virasoro-type extensions for the higman--thompson and neretin groups.
\newblock {\em The Quarterly Journal of Mathematics}, 53(3):295--317, 2002.

\bibitem{KapoudjianSergiescu2005}
C.~Kapoudjian and V.~Sergiescu.
\newblock An extension of the {B}urau representation to a mapping class group
  associated to {T}hompson's group {$T$}.
\newblock In {\em Geometry and dynamics}, volume 389 of {\em Contemp. Math.},
  pages 141--164. Amer. Math. Soc., Providence, RI, 2005.

\bibitem{matui2015topological}
H.~Matui.
\newblock Topological full groups of one-sided shifts of finite type.
\newblock {\em Journal f{\"u}r die reine und angewandte Mathematik (Crelles
  Journal)}, 2015(705):35--84, 2015.

\bibitem{nekrashevych_2019}
V.~Nekrashevych.
\newblock Simple groups of dynamical origin.
\newblock {\em Ergodic Theory and Dynamical Systems}, 39(3):707--732, 2019.

\bibitem{Vitali}
G.~Vitali.
\newblock {\em Sostituzioni sopra un'infinita numerabile di elementi}.
\newblock Tip. Fusi, 1915.

\end{thebibliography}
\end{document}